\documentclass[11pt]{article}

\usepackage[a4paper,margin=1in]{geometry}
\usepackage{amsmath,amssymb,amsthm,mathtools}
\usepackage{aliascnt}
\usepackage{microtype}
\usepackage{enumitem}
\usepackage[numbers,sort&compress]{natbib}
\usepackage[colorlinks=true,linkcolor=blue,citecolor=blue,urlcolor=blue]{hyperref}

\allowdisplaybreaks
\newcommand{\R}{\mathbb{R}}
\newcommand{\N}{\mathbb{N}}
\renewcommand{\P}{\mathcal{P}}
\newcommand{\M}{\mathcal{M}}
\newcommand{\A}{\mathcal{A}}
\newcommand{\OT}{\mathsf{OT}}
\newcommand{\BC}{\mathsf{BC}}
\newcommand{\CCOT}{\mathsf{CCOT}}
\newcommand{\CCBC}{\mathsf{CCBC}}
\newcommand{\EOT}{\mathsf{EOT}}
\newcommand{\EBC}{\mathsf{EBC}}
\newcommand{\ECCOT}{\mathsf{ECCOT}}
\newcommand{\ECCBC}{\mathsf{ECCBC}}
\newcommand{\KL}{\operatorname{KL}}
\newcommand{\spt}{\operatorname{spt}}
\renewcommand{\d}{\,\mathrm{d}}
\newcommand{\pos}[1]{\left[#1\right]_{+}}
\DeclareMathOperator{\Proj}{Proj}

\theoremstyle{plain}
\newtheorem{theorem}{Theorem}[section]

\newaliascnt{lemma}{theorem}
\newtheorem{lemma}[lemma]{Lemma}
\aliascntresetthe{lemma}

\newaliascnt{proposition}{theorem}
\newtheorem{proposition}[proposition]{Proposition}
\aliascntresetthe{proposition}

\newaliascnt{corollary}{theorem}
\newtheorem{corollary}[corollary]{Corollary}
\aliascntresetthe{corollary}

\theoremstyle{definition}
\newaliascnt{definition}{theorem}
\newtheorem{definition}[definition]{Definition}
\aliascntresetthe{definition}

\newaliascnt{assumption}{theorem}
\newtheorem{assumption}[assumption]{Assumption}
\aliascntresetthe{assumption}

\theoremstyle{remark}
\newaliascnt{remark}{theorem}
\newtheorem{remark}[remark]{Remark}
\aliascntresetthe{remark}

\newaliascnt{example}{theorem}

\aliascntresetthe{example}

\usepackage[nameinlink,noabbrev]{cleveref}
\crefname{assumption}{Assumption}{Assumptions}
\Crefname{assumption}{Assumption}{Assumptions}
\crefname{definition}{Definition}{Definitions}
\Crefname{definition}{Definition}{Definitions}
\crefname{remark}{Remark}{Remarks}
\Crefname{remark}{Remark}{Remarks}
\crefname{example}{Example}{Examples}
\Crefname{example}{Example}{Examples}

\title{Capacity-Constrained Wasserstein Barycenters:\\
Existence, Duality, and Entropic Regularization}
\author{Chamila Malagoda Gamage}
\date{}
\hypersetup{
  pdftitle={Capacity-Constrained Wasserstein Barycenters: Existence, Duality, and Entropic Regularization},
  pdfauthor={Chamila Malagoda Gamage},
  pdfsubject={Capacity-constrained Wasserstein barycenters and entropy-regularized duality},
  pdfkeywords={optimal transport, Wasserstein barycenter, capacity constraint, entropy regularization, convex duality}
}

\begin{document}
\maketitle

\begin{abstract}
We introduce a capacity-constrained Wasserstein barycenter problem in
which each transport plan from an input measure to the barycenter is
bounded by a prescribed capacity measure.  On compact domains, we
prove existence of capacity-constrained barycenters and establish a
strong duality formula.  We then study an entropy-regularized version
under bounded capacity-density assumptions.  The regularized problem
has a unique optimal tuple of transport plans and a unique barycenter,
while its dual involves an explicit capped-exponential penalty.
Whenever dual maximizers exist, the optimal densities satisfy a
capacity-clipped Gibbs formula.  Finally, we prove an
$O(\epsilon)$ estimate for the optimal values and show that the
regularized plans converge to the minimum-entropy optimal solution of
the unregularized problem.
\end{abstract}


\section{Introduction}\label{sec:introduction}

The Wasserstein barycenter provides a nonlinear interpolation and averaging operation for probability measures.  Given probability measures $\nu_1,\ldots,\nu_p$ and positive weights $\lambda_1,\ldots,\lambda_p$ with $\sum_i\lambda_i=1$, the classical quadratic barycenter minimizes a weighted sum of transport costs from the input measures to a common probability measure; see \cite{agueh2011barycenters}.  Wasserstein barycenters have become useful in probability, statistics, imaging, economics, and machine learning, but the classical formulation permits every coupling with the prescribed marginals.

In many transport systems, not every coupling is admissible.  A road, communication channel, storage facility, matching mechanism, or production network may impose an upper bound on the mass that can pass through a portion of the product space.  Capacity-constrained optimal transport models this restriction by requiring the transport plan to be dominated by a prescribed finite measure.  The two-marginal theory, including structural properties and duality, is developed in \cite{rachev1998mass,korman2015optimal,korman2015dual}.  Other notions of constrained barycenters have been considered for image priors and related computational purposes \cite{simon2020barycenters}, but these constraints act on the barycenter itself and differ from the pairwise capacity restrictions considered here.  The phrase ``capacity-constrained Wasserstein barycenter'' has also appeared in semi-discrete computational geometry, where prescribed masses are assigned to the cells of a centroidal power diagram \cite{xin2016centroidal}.  That cell-capacity model is distinct from the present measure-theoretic formulation, in which each full coupling is required to satisfy a domination constraint on the product space.

The main object of this paper is the \emph{capacity-constrained barycenter} problem.  A separate capacity is imposed on each coupling between an input measure and the unknown barycenter.  Thus the admissible barycenter is not merely required to belong to a prescribed set: it must be simultaneously reachable from every input measure through a coupling that respects the corresponding transport capacity.

Our first contribution is a systematic primal and dual theory for this model.  We characterize the feasible set through the two-marginal capacity criterion, show that the feasible barycenter set and the objective are convex, identify the support restriction imposed by the capacities, prove weak compactness and existence, and obtain strong duality.  The dual variables have a direct interpretation: the usual transport potentials enforce the marginal constraints, while nonpositive capacity potentials encode the shadow price of saturating the available transport capacity.

Our second contribution is an entropy-regularized theory that retains these same capacity constraints.  Following the continuous regularized-barycenter viewpoint of \cite{li2020continuous}, we introduce a fixed support measure $\eta$ for the unknown barycenter.  Relative to the product reference $\nu_i\otimes\eta$, the $i$th capacity is represented by a bounded density $h_i$.  Entropy regularization then acts on densities $q_i$ satisfying the hard upper bound $0\le q_i\le h_i$.  The resulting regularized barycenter is unique, and strong value duality follows from a compact minimax formulation.  The key dual quantity is the capped-exponential conjugate
\[
\Phi_{\epsilon,h}(s)
=\sup_{0\le q\le h}\bigl\{sq-\epsilon(q\log q-q)\bigr\}.
\]
It coincides with the usual exponential conjugate before the capacity becomes active and becomes affine after saturation.  Whenever the dual supremum is attained, the corresponding optimal density satisfies
\[
q_i(x,y)
=\min\left\{h_i(x,y),
\exp\!\left(\frac{\phi_i(x)+\psi_i(y)-c_i(x,y)}{\epsilon}\right)\right\}.
\]
This is a capacity-clipped analogue of the Gibbs relation in entropy-regularized optimal transport.  We also prove that the regularized values converge at order $O(\epsilon)$ and that the full family of regularized optimal plans converges to the unique entropy-minimizing unregularized optimal tuple.

Entropy regularization of optimal transport underlies Sinkhorn-type algorithms \cite{cuturi2013sinkhorn,peyre2019computational}.  Continuous entropic duality is studied in \cite{clason2021entropic}, while numerical regularizations and Sinkhorn-type methods for constrained transport are developed in \cite{wu2023double,tang2024sinkhorn}.  Recent work on regularized barycenters includes \cite{chizat2025doubly}.  The present paper connects these two lines of research by regularizing the capacity-constrained barycenter itself and retaining the capacity multipliers in the dual problem.

The paper is organized as follows.  \Cref{sec:preliminaries} recalls classical and capacity-constrained optimal transport, together with their dual formulations.  \Cref{sec:ccbc} first recalls the classical Wasserstein barycenter problem and its dual, and then develops the primal and dual theory of capacity-constrained barycenters.  \Cref{sec:entropic} reviews classical entropy-regularized transport and fixed-support regularized barycenters before introducing entropy-regularized capacity-constrained transport and barycenters.  \Cref{sec:zero-limit} proves convergence and entropy selection as $\epsilon\downarrow0$.  An auxiliary moment estimate on $\R^d$ is included in \Cref{sec:appendix-moments}.

\section{Classical and capacity-constrained optimal transport}
\label{sec:preliminaries}

Let $K\subset\R^d$ be compact.  We denote by $\P(K)$ the set of Borel
probability measures on $K$, by $\M_+(K)$ the set of finite nonnegative Borel
measures on $K$, and by $C(K)$ the space of real-valued continuous functions
on $K$.  Similarly, $C(K\times K)$ denotes the space of real-valued continuous
functions on $K\times K$.  These function spaces are equipped with the uniform
norm.  Weak convergence of measures is denoted by $\rightharpoonup$.  For
$\gamma,\widetilde\gamma\in\M_+(K\times K)$, the notation
$\gamma\le\widetilde\gamma$ means that
$\widetilde\gamma-\gamma$ is a nonnegative measure.

Let $\mu,\nu\in\P(K)$, and let
$\pi_1,\pi_2:K\times K\to K$ denote the coordinate projections,
\[
\pi_1(x,y)=x,
\qquad
\pi_2(x,y)=y.
\]

\begin{definition}[Transport plans and classical optimal transport]
\label{def:classical-ot}
A probability measure $\gamma\in\P(K\times K)$ is called a
\emph{transport plan}, or equivalently a \emph{coupling}, from $\mu$ to $\nu$
if
\[
(\pi_1)_\#\gamma=\mu,
\qquad
(\pi_2)_\#\gamma=\nu,
\]
where $T_\#\rho$ denotes the push-forward of a measure $\rho$ by a Borel map
$T$.  The set of all transport plans from $\mu$ to $\nu$ is denoted by
\begin{equation}\label{eq:transport-plans}
\Pi(\mu,\nu)
:=
\left\{
\gamma\in\P(K\times K):
(\pi_1)_\#\gamma=\mu,
\ (\pi_2)_\#\gamma=\nu
\right\}.
\end{equation}
The two equalities in \eqref{eq:transport-plans} are called the
\emph{marginal conditions}.  Equivalently,
\begin{equation}\label{eq:marginal-conditions}
\gamma(A\times K)=\mu(A),
\qquad
\gamma(K\times B)=\nu(B),
\end{equation}
for every pair of Borel sets $A,B\subset K$.

For a continuous cost $c:K\times K\to\R$, the classical Kantorovich optimal
transport value is
\begin{equation}\label{eq:classical-ot}
\OT_c(\mu,\nu)
:=
\inf_{\gamma\in\Pi(\mu,\nu)}
\int_{K\times K}c(x,y)\d\gamma(x,y).
\end{equation}
Any minimizer is called an \emph{optimal transport plan}.
\end{definition}

The set $\Pi(\mu,\nu)$ is nonempty because
$\mu\otimes\nu\in\Pi(\mu,\nu)$.  It is weakly compact, and hence the infimum
in \eqref{eq:classical-ot} is attained.

For later comparison with the capacity-constrained problem, define the
classical Kantorovich dual value by
\begin{align}\label{eq:classical-ot-dual}
\OT_c^*(\mu,\nu)
:=
\sup\Bigg\{\,
&\int_Ku\d\mu+\int_Kv\d\nu:\nonumber\\
&u,v\in C(K),\qquad
u(x)+v(y)\le c(x,y)
\quad\text{for all }(x,y)\in K\times K
\,\Bigg\}.
\end{align}

\begin{theorem}[Classical Kantorovich duality]
\label{thm:classical-ot-duality}
For every continuous cost $c:K\times K\to\R$,
\begin{equation}\label{eq:classical-ot-strong-duality}
\OT_c(\mu,\nu)=\OT_c^*(\mu,\nu).
\end{equation}
Moreover, the supremum in \eqref{eq:classical-ot-dual} is attained.  A
maximizing pair is called a pair of \emph{Kantorovich potentials}.
\end{theorem}

The weak-duality inequality follows by integrating
$u(x)+v(y)\le c(x,y)$ against any $\gamma\in\Pi(\mu,\nu)$.  The reverse
inequality and dual attainment are the classical Kantorovich duality theorem;
see \cite{villani2009optimal}.

The capacity-constrained formulation is obtained by restricting the classical
admissible set $\Pi(\mu,\nu)$ to transport plans dominated by a prescribed
finite measure.  Let $\widetilde\gamma\in\M_+(K\times K)$.

\begin{definition}[Capacity-constrained couplings]
\label{def:capacity-couplings}
The set of transport plans from $\mu$ to $\nu$ bounded by
$\widetilde\gamma$ is
\begin{equation}\label{eq:capacity-couplings}
\Pi^{\widetilde\gamma}(\mu,\nu)
:=\bigl\{\gamma\in\Pi(\mu,\nu):\gamma\le\widetilde\gamma\bigr\}.
\end{equation}
For a continuous cost $c:K\times K\to\R$, the capacity-constrained transport
value is
\begin{equation}\label{eq:ccot-primal}
\CCOT_c^{\widetilde\gamma}(\mu,\nu)
:=\inf_{\gamma\in\Pi^{\widetilde\gamma}(\mu,\nu)}
\int_{K\times K}c(x,y)\d\gamma(x,y),
\end{equation}
with the convention that the value is $+\infty$ if the admissible set is
empty.
\end{definition}

The following criterion describes when the capacity and marginal constraints
are compatible.

\begin{theorem}[Capacity feasibility criterion]
\label{thm:capacity-feasibility}
Let $\mu,\nu\in\P(K)$ and
$\widetilde\gamma\in\M_+(K\times K)$.  Then
$\Pi^{\widetilde\gamma}(\mu,\nu)\neq\varnothing$ if and only if
\begin{equation}\label{eq:capacity-feasibility}
\mu(A)+\nu(B)
\le \widetilde\gamma(A\times B)+1
\qquad
\text{for all Borel sets }A,B\subset K.
\end{equation}
\end{theorem}

This is the compact-space form of
\cite[Corollary~4.6.15]{rachev1998mass}.

The dual of the capacity-constrained transport problem extends the classical
Kantorovich dual by introducing a nonpositive multiplier for the domination
constraint $\gamma\le\widetilde\gamma$.  Define
\begin{align}\label{eq:ccot-dual}
\CCOT_c^{\widetilde\gamma,*}(\mu,\nu)
:=\sup\Bigg\{\,
&\int_Ku\d\mu+\int_Kv\d\nu
 +\int_{K\times K}w\d\widetilde\gamma:\nonumber\\
&u,v\in C(K),\quad w\in C(K\times K),\quad w\le0,\nonumber\\
&u(x)+v(y)+w(x,y)\le c(x,y)
  \quad\text{for all }(x,y)\in K\times K
\,\Bigg\}.
\end{align}

\begin{theorem}[Two-marginal capacity-constrained duality]
\label{thm:ccot-duality}
For every continuous cost $c:K\times K\to\R$,
\begin{equation}\label{eq:ccot-strong-duality}
\CCOT_c^{\widetilde\gamma}(\mu,\nu)
=\CCOT_c^{\widetilde\gamma,*}(\mu,\nu).
\end{equation}
\end{theorem}

The result follows from the capacity-constrained Kantorovich duality of Rachev
and R\"uschendorf \cite{rachev1998mass}; see also
\cite{korman2015dual,korman2015optimal}.

\begin{remark}[Sign of the capacity potential]
\label{rem:capacity-sign}
The condition $w\le0$ is forced by the upper-bound constraint.  Indeed, if
$\gamma\le\widetilde\gamma$, then
$\int w\d\widetilde\gamma\le\int w\d\gamma$ for every nonpositive $w$.
Thus $w$ lowers the dual value precisely on regions in which the prescribed
capacity may become active.
\end{remark}

\section{The capacity-constrained barycenter problem}
\label{sec:ccbc}

Throughout this section, let $p\ge2$, let
$\nu_1,\ldots,\nu_p\in\P(K)$, and let
$\lambda_1,\ldots,\lambda_p>0$ satisfy
$\sum_{i=1}^p\lambda_i=1$.  For each $i$, let
$\widetilde\gamma_i\in\M_+(K\times K)$ be a prescribed capacity.  We take
\begin{equation}\label{eq:quadratic-cost}
c_i(x,y):=\frac12\lvert x-y\rvert^2,
\end{equation}
but the compact-space arguments below apply to arbitrary continuous costs.

Before imposing capacity constraints, we briefly recall the classical
Wasserstein barycenter problem.  It may be viewed as a nonlinear interpolation
among several probability measures; see \cite{agueh2011barycenters}.

\begin{definition}[Classical Wasserstein barycenter]
\label{def:classical-barycenter}
The classical Wasserstein barycenter value associated with the measures
$\nu_1,\ldots,\nu_p$ and weights $\lambda_1,\ldots,\lambda_p$ is
\begin{equation}\label{eq:classical-barycenter}
\BC
:=\inf_{\nu\in\P(K)}J_0(\nu),
\qquad
J_0(\nu)
:=\sum_{i=1}^p\lambda_i\OT_{c_i}(\nu_i,\nu).
\end{equation}
Equivalently,
\begin{equation}\label{eq:classical-barycenter-expanded}
\BC
=
\inf_{\nu\in\P(K)}
\sum_{i=1}^p\lambda_i
\inf_{\gamma_i\in\Pi(\nu_i,\nu)}
\int_{K\times K}\frac12\lvert x-y\rvert^2\d\gamma_i(x,y).
\end{equation}
Any minimizer is called a \emph{Wasserstein barycenter} of
$\nu_1,\ldots,\nu_p$ with weights $\lambda_1,\ldots,\lambda_p$.
\end{definition}

Since $K$ is compact and the costs $c_i$ are continuous, the classical
barycenter problem admits at least one minimizer.  In this formulation, every
transport plan satisfying the prescribed marginal conditions is admissible.
The capacity-constrained problem introduced below replaces each set
$\Pi(\nu_i,\nu)$ by $\Pi^{\widetilde\gamma_i}(\nu_i,\nu)$.  Consequently, the
common measure $\nu$ must be simultaneously reachable from every input measure
through transport plans satisfying their respective capacity constraints.

\subsection{Primal problem and elementary properties}\label{subsec:ccbc-primal}

\begin{definition}[Capacity-constrained barycenter]\label{def:ccbc}
The feasible barycenter set is
\begin{equation}\label{eq:feasible-barycenters}
\P_{\widetilde\gamma}
:=\left\{\nu\in\P(K):
\Pi^{\widetilde\gamma_i}(\nu_i,\nu)\neq\varnothing
\text{ for every }i=1,\ldots,p\right\}.
\end{equation}
For $\nu\in\P(K)$, define
\begin{equation}\label{eq:pairwise-cc-cost}
\widetilde W_i^2(\nu_i,\nu)
:=\CCOT_{c_i}^{\widetilde\gamma_i}(\nu_i,\nu).
\end{equation}
The capacity-constrained barycenter value is
\begin{equation}\label{eq:ccbc-primal}
\CCBC
:=\inf_{\nu\in\P_{\widetilde\gamma}}
J(\nu),
\qquad
J(\nu):=\sum_{i=1}^p\lambda_i\widetilde W_i^2(\nu_i,\nu).
\end{equation}
Any minimizer is called a capacity-constrained barycenter of
$\nu_1,\ldots,\nu_p$ associated with the capacities
$\widetilde\gamma_1,\ldots,\widetilde\gamma_p$.
\end{definition}

We assume from now on that $\P_{\widetilde\gamma}\neq\varnothing$.  This is a compatibility condition on the family of capacities.  By \Cref{thm:capacity-feasibility}, it may be checked pairwise for a candidate barycenter.

\begin{remark}[A simple feasible family]\label{rem:simple-feasible-family}
The nonempty condition is not restrictive in principle.  For example, if $\xi\in\P(K)$ and
$\widetilde\gamma_i=\nu_i\otimes\xi$, then $\xi\in\P_{\widetilde\gamma}$ because
$\nu_i\otimes\xi\in\Pi^{\widetilde\gamma_i}(\nu_i,\xi)$.
\end{remark}

\begin{proposition}[Support, convexity, and monotonicity]\label{prop:basic-properties}
The capacity-constrained barycenter problem has the following properties.
\begin{enumerate}[label=\textup{(\roman*)}]
\item Every $\nu\in\P_{\widetilde\gamma}$ satisfies
\begin{equation}\label{eq:support-control}
\spt\nu\subset
\bigcap_{i=1}^p
\overline{\Proj_y(\spt\widetilde\gamma_i)}.
\end{equation}
\item The set $\P_{\widetilde\gamma}$ is convex.
\item For each $i$, the map
$\nu\mapsto\widetilde W_i^2(\nu_i,\nu)$ is convex on
$\P_{\widetilde\gamma}$; hence $J$ is convex.
\item If $\widetilde\gamma_i\le\widehat\gamma_i$ for every $i$, then
\begin{equation}\label{eq:capacity-monotonicity}
\P_{\widetilde\gamma}\subset\P_{\widehat\gamma}
\qquad\text{and}\qquad
\CCBC(\widehat\gamma_1,\ldots,\widehat\gamma_p)
\le
\CCBC(\widetilde\gamma_1,\ldots,\widetilde\gamma_p).
\end{equation}
\end{enumerate}
\end{proposition}

\begin{proof}
If $\gamma_i\in\Pi^{\widetilde\gamma_i}(\nu_i,\nu)$, then the second marginal of $\gamma_i$ is supported on the closure of the second projection of $\spt\widetilde\gamma_i$, proving \eqref{eq:support-control}.

Let $\nu^0,\nu^1\in\P_{\widetilde\gamma}$ and $t\in[0,1]$.  Choose
$\gamma_i^k\in\Pi^{\widetilde\gamma_i}(\nu_i,\nu^k)$ for $k=0,1$.  Then
$t\gamma_i^0+(1-t)\gamma_i^1$ has first marginal $\nu_i$, second marginal
$t\nu^0+(1-t)\nu^1$, and is still dominated by $\widetilde\gamma_i$.  This proves convexity of the feasible set.  Taking $\gamma_i^0$ and $\gamma_i^1$ optimal, or arbitrarily close to optimal, and using linearity of the cost gives convexity of each constrained cost and of $J$.

Finally, increasing a capacity can only enlarge the corresponding admissible set of plans.  The inclusion of feasible barycenter sets and the inequality of optimal values follow immediately.
\end{proof}

\begin{lemma}[Weak lower semicontinuity]\label{lem:cc-lsc}
For each $i$, the map
$\nu\mapsto\widetilde W_i^2(\nu_i,\nu)$ is weakly lower semicontinuous on $\P(K)$, with the value $+\infty$ outside its feasible domain.
\end{lemma}

\begin{proof}
Let $\nu_n\rightharpoonup\nu$ and suppose that
$\liminf_n\widetilde W_i^2(\nu_i,\nu_n)<\infty$.  Passing to a subsequence, we may assume that the liminf is a limit.  Choose optimal plans
$\gamma_{i,n}\in\Pi^{\widetilde\gamma_i}(\nu_i,\nu_n)$.  Since $K\times K$ is compact, a further subsequence converges weakly to some
$\gamma_i\in\P(K\times K)$.  The marginal constraints pass to the limit, so
$\gamma_i\in\Pi(\nu_i,\nu)$.

For every nonnegative $f\in C(K\times K)$,
\[
\int f\d\gamma_i
=\lim_n\int f\d\gamma_{i,n}
\le\int f\d\widetilde\gamma_i.
\]
Hence $\gamma_i\le\widetilde\gamma_i$.  Since $c_i$ is continuous,
\[
\widetilde W_i^2(\nu_i,\nu)
\le\int c_i\d\gamma_i
=\lim_n\int c_i\d\gamma_{i,n}
=\liminf_n\widetilde W_i^2(\nu_i,\nu_n).
\]
\end{proof}

\begin{proposition}[Compactness of the feasible set]\label{prop:feasible-compact}
The set $\P_{\widetilde\gamma}$ is weakly compact in $\P(K)$.
\end{proposition}

\begin{proof}
Since $K$ is compact, it is enough to prove that the feasible set is weakly closed.  Let
$\nu_n\in\P_{\widetilde\gamma}$ and $\nu_n\rightharpoonup\nu$.  For each $i$ and $n$, choose
$\gamma_{i,n}\in\Pi^{\widetilde\gamma_i}(\nu_i,\nu_n)$.  By compactness and a diagonal extraction, assume that
$\gamma_{i,n}\rightharpoonup\gamma_i$ for every $i$.  As in the proof of
\Cref{lem:cc-lsc}, the limiting plan belongs to
$\Pi^{\widetilde\gamma_i}(\nu_i,\nu)$.  Therefore $\nu\in\P_{\widetilde\gamma}$.
\end{proof}

\begin{theorem}[Existence of a capacity-constrained barycenter]\label{thm:ccbc-existence}
If $\P_{\widetilde\gamma}\neq\varnothing$, then the problem \eqref{eq:ccbc-primal} has a minimizer.
\end{theorem}

\begin{proof}
Let $\{\nu_n\}\subset\P_{\widetilde\gamma}$ be a minimizing sequence.  By
\Cref{prop:feasible-compact}, a subsequence converges weakly to some
$\nu^*\in\P_{\widetilde\gamma}$.  By \Cref{lem:cc-lsc},
\[
J(\nu^*)
\le\liminf_{n\to\infty}J(\nu_n)
=\CCBC.
\]
Thus $\nu^*$ is a minimizer.
\end{proof}

\subsection{Strong duality}
\label{subsec:ccbc-duality}

The capacity-constrained dual is most transparent when compared with the dual
of the classical Wasserstein barycenter problem.  Define
\begin{equation}\label{eq:unregularized-potential-class}
\A_0
:=
\left\{
(\boldsymbol\phi,\boldsymbol\psi):
\begin{array}{l}
\phi_i,\psi_i\in C(K),\quad i=1,\ldots,p,\\[1mm]
\displaystyle\sum_{i=1}^p\lambda_i\psi_i=0
\quad\text{on }K
\end{array}
\right\}.
\end{equation}
The weighted zero-sum condition expresses the fact that all pairwise transport
problems have the same second marginal.

Define the classical barycenter admissible potential class by
\begin{equation}\label{eq:classical-barycenter-potential-class}
\A_{\mathrm{cl}}
:=
\left\{
(\boldsymbol\phi,\boldsymbol\psi)\in\A_0:
\phi_i(x)+\psi_i(y)\le c_i(x,y)
\ \text{for all }(x,y)\in K\times K,
\ i=1,\ldots,p
\right\}.
\end{equation}

\begin{theorem}[Classical Wasserstein barycenter duality]
\label{thm:classical-barycenter-duality}
The classical barycenter value in \eqref{eq:classical-barycenter} satisfies
\begin{equation}\label{eq:classical-barycenter-dual}
\BC
=
\sup_{(\boldsymbol\phi,\boldsymbol\psi)\in\A_{\mathrm{cl}}}
\sum_{i=1}^p\lambda_i\int_K\phi_i\d\nu_i.
\end{equation}
For the quadratic costs \eqref{eq:quadratic-cost}, the supremum is attained.
\end{theorem}

This is the compact-domain form of the classical Wasserstein barycenter
duality of Agueh and Carlier \cite{agueh2011barycenters}.  To see the role of
the zero-sum condition, let $\nu\in\P(K)$ and
$\gamma_i\in\Pi(\nu_i,\nu)$.  For every
$(\boldsymbol\phi,\boldsymbol\psi)\in\A_{\mathrm{cl}}$,
\[
\int_K\phi_i\d\nu_i+\int_K\psi_i\d\nu
\le\int_{K\times K}c_i\d\gamma_i.
\]
Multiplying by $\lambda_i$, summing over $i$, and using
$\sum_i\lambda_i\psi_i=0$ gives the classical weak-duality inequality.  Strong
duality states that this lower bound is sharp.

Equivalently, setting $f_i:=\lambda_i\psi_i$ and eliminating the source
potentials gives the infimum-transform form
\begin{align}\label{eq:classical-barycenter-dual-inf-transform}
\BC
=
\sup\Bigg\{&
\sum_{i=1}^p
\int_K
\inf_{y\in K}
\left\{\lambda_i c_i(x,y)-f_i(y)\right\}
\d\nu_i(x):\nonumber\\
&f_i\in C(K),
\qquad
\sum_{i=1}^pf_i=0
\quad\text{on }K
\Bigg\}.
\end{align}

The marginal structure of the capacity-constrained barycenter is unchanged.
The new feature is that each domination constraint
$\gamma_i\le\widetilde\gamma_i$ introduces a nonpositive capacity multiplier.
Eliminating this multiplier replaces the classical pointwise constraint
$\phi_i+\psi_i\le c_i$ by a positive-part penalty integrated against the
prescribed capacity.  This yields the following theorem and provides the
natural bridge to entropy regularization.

\begin{theorem}[Strong duality for capacity-constrained barycenters]\label{thm:ccbc-duality}
If \(\P_{\widetilde\gamma}\neq\varnothing\), then
\begin{align}\label{eq:ccbc-dual-positive-part}
\CCBC
=\sup_{(\boldsymbol\phi,\boldsymbol\psi)\in\A_0}
\sum_{i=1}^p\lambda_i
\left[
\int_K\phi_i\d\nu_i
-\int_{K\times K}
\pos{\phi_i(x)+\psi_i(y)-c_i(x,y)}
\d\widetilde\gamma_i(x,y)
\right].
\end{align}
\end{theorem}

\begin{proof}
Every measure \(\gamma_i\le\widetilde\gamma_i\) has the form
\(\gamma_i=q_i\widetilde\gamma_i\) for some
\[
q_i\in\mathcal D_i
:=\{q\in L^\infty(K\times K,\widetilde\gamma_i):0\le q\le1\}.
\]
Endow \(\mathcal D_i\) with the weak-* topology of \(L^\infty\).  It is convex and weak-* compact.  A tuple \(\boldsymbol q=(q_1,\ldots,q_p)\) represents an admissible barycenter tuple precisely when the first marginal of \(q_i\widetilde\gamma_i\) is \(\nu_i\) for every \(i\), and the second marginals of all the measures \(q_i\widetilde\gamma_i\) coincide.

These linear constraints have the indicator representation
\begin{align}\label{eq:barycentric-constraint-indicator}
\mathcal I(\boldsymbol q)
=\sup_{(\boldsymbol\phi,\boldsymbol\psi)\in\A_0}
\sum_{i=1}^p\lambda_i
\left[
\int_K\phi_i\d\nu_i
-\int_{K\times K}
\bigl(\phi_i(x)+\psi_i(y)\bigr)q_i(x,y)
\d\widetilde\gamma_i(x,y)
\right],
\end{align}
where \(\mathcal I(\boldsymbol q)=0\) for an admissible tuple and
\(\mathcal I(\boldsymbol q)=+\infty\) otherwise.  Indeed, the expression vanishes on every admissible tuple.  If a first marginal is incorrect, a source potential \(\phi_i\) separates the two measures; if the second marginals do not coincide, continuous functions \(\psi_i\) satisfying the weighted zero-sum condition separate them.  Scaling the separating functions gives \(+\infty\).

Consequently,
\begin{align*}
\CCBC
&=\min_{\boldsymbol q\in\prod_i\mathcal D_i}
\sup_{(\boldsymbol\phi,\boldsymbol\psi)\in\A_0}
\sum_{i=1}^p\lambda_i
\left[
\int_K\phi_i\d\nu_i
+\int_{K\times K}
\bigl(c_i-\phi_i-\psi_i\bigr)q_i
\d\widetilde\gamma_i
\right].
\end{align*}
The set \(\prod_i\mathcal D_i\) is compact and convex, the displayed functional is affine and weak-* continuous in \(\boldsymbol q\), and it is affine in the potentials.  Sion's minimax theorem \cite{sion1958general} therefore permits the interchange of minimum and supremum.  Finally, the minimization in \(q_i\) is pointwise, and
\[
\min_{0\le q\le1}(c_i-\phi_i-\psi_i)q
=-\pos{\phi_i+\psi_i-c_i}.
\]
Substitution gives \eqref{eq:ccbc-dual-positive-part}.
\end{proof}

The positive-part penalty is equivalent to a formulation with explicit capacity potentials.  This is the form closest to the two-marginal duality in \Cref{thm:ccot-duality}.

\begin{corollary}[Capacity-multiplier form]\label{cor:ccbc-capacity-multiplier}
Let \(\widehat\A_0\) be the set of
\(\phi_i,\psi_i\in C(K)\) and
\(w_i\in C(K\times K)\) satisfying
\begin{equation}\label{eq:unreg-dual-constraints}
w_i\le0,
\qquad
\phi_i(x)+\psi_i(y)+w_i(x,y)\le c_i(x,y),
\qquad
\sum_{i=1}^p\lambda_i\psi_i=0.
\end{equation}
Then
\begin{equation}\label{eq:ccbc-dual-symmetric}
\CCBC
=\sup_{(\boldsymbol\phi,\boldsymbol\psi,\boldsymbol w)\in\widehat\A_0}
\sum_{i=1}^p\lambda_i
\left(
\int_K\phi_i\d\nu_i
+\int_{K\times K}w_i\d\widetilde\gamma_i
\right).
\end{equation}
\end{corollary}

\begin{proof}
For fixed \(\phi_i\) and \(\psi_i\), the largest admissible capacity potential is
\[
w_i(x,y)=-\pos{\phi_i(x)+\psi_i(y)-c_i(x,y)}.
\]
It is continuous, nonpositive, and satisfies the pointwise constraint in
\eqref{eq:unreg-dual-constraints}.  Substituting it into
\eqref{eq:ccbc-dual-symmetric} yields
\eqref{eq:ccbc-dual-positive-part}; every other admissible \(w_i\) gives a no larger value.
\end{proof}

Eliminating the source potentials recovers the form used in the original capacity-constrained barycenter argument.

\begin{corollary}[Infimum-transform form of the dual]\label{cor:ccbc-inf-transform}
Let \(f_i:=\lambda_i\psi_i\) and \(\omega_i:=\lambda_i w_i\).  Then
\begin{align}\label{eq:ccbc-dual-inf-transform}
\CCBC
=\sup\Bigg\{&\sum_{i=1}^p
\int_K
\inf_{z\in K}
\left\{\lambda_i c_i(x,z)-f_i(z)-\omega_i(x,z)\right\}
\d\nu_i(x)\nonumber\\
&\quad+
\sum_{i=1}^p\int_{K\times K}\omega_i\d\widetilde\gamma_i:
\quad
f_i\in C(K),\quad
\omega_i\in C(K\times K),\quad
\omega_i\le0,
\quad
\sum_{i=1}^pf_i=0
\Bigg\}.
\end{align}
\end{corollary}

\begin{proof}
For fixed \(\psi_i\) and \(w_i\), the largest admissible source potential is
\[
\phi_i(x)
=\inf_{z\in K}\{c_i(x,z)-\psi_i(z)-w_i(x,z)\}.
\]
Substituting this expression into \eqref{eq:ccbc-dual-symmetric} and multiplying the \(i\)th potentials by \(\lambda_i\) gives \eqref{eq:ccbc-dual-inf-transform}.
\end{proof}

\begin{remark}[Recovery of the classical barycenter dual]
\label{rem:classical-special-case}
Setting $w_i\equiv0$ in the capacity-multiplier constraints
\eqref{eq:unreg-dual-constraints} recovers the classical admissible potential
class \eqref{eq:classical-barycenter-potential-class}.  The objective in
\eqref{eq:ccbc-dual-symmetric} then reduces to the classical barycenter dual
objective \eqref{eq:classical-barycenter-dual}.  Thus the functions $w_i$
measure precisely the departure from the unconstrained barycenter geometry.

Moreover, if a classical optimal barycenter admits optimal transport plans
satisfying all prescribed capacity bounds, then it is feasible for the
capacity-constrained problem and the two optimal values coincide.
\end{remark}

\section{Entropy regularization: classical and capacity-constrained models}
\label{sec:entropic}

We first recall the classical entropy-regularized transport problem and the
fixed-support regularized barycenter model.  We then add the hard transport
capacities introduced in the preceding sections.

Subsections~\ref{subsec:classical-entropic-ot} and
\ref{subsec:classical-entropic-barycenter} collect established
background on entropy-regularized transport and fixed-support
regularized barycenters.  The analysis specific to the present
capacity-constrained model begins in
Subsection~\ref{subsec:entropic-capacity-transport}.

\subsection{Classical entropy-regularized optimal transport}
\label{subsec:classical-entropic-ot}

\begin{definition}[Relative entropy]\label{def:relative-entropy}
For $\gamma,\xi\in\P(K\times K)$, the relative entropy of $\gamma$ with
respect to $\xi$ is
\begin{equation}\label{eq:kl-definition}
\KL(\gamma\mid\xi)
:=
\begin{cases}
\displaystyle
\int_{K\times K}
\log\!\left(\frac{\d\gamma}{\d\xi}\right)\d\gamma,
&\gamma\ll\xi,\\[0.8em]
+\infty,&\text{otherwise}.
\end{cases}
\end{equation}
If $\gamma=q\xi$, then
\begin{equation}\label{eq:entropy-shift}
\KL(\gamma\mid\xi)
=1+\int_{K\times K}\rho(q)\d\xi,
\qquad
\rho(q):=q\log q-q,
\end{equation}
where $0\log0:=0$.
\end{definition}

\begin{definition}[Entropy-regularized optimal transport]
\label{def:classical-entropic-ot}
Let $\mu,\nu\in\P(K)$, let $c\in C(K\times K)$, and let
$\epsilon>0$.  The entropy-regularized optimal transport value is
\begin{equation}\label{eq:classical-entropic-ot}
\EOT_{\epsilon,c}(\mu,\nu)
:=
\inf_{\gamma\in\Pi(\mu,\nu)}
\left\{
\int_{K\times K}c\d\gamma
+\epsilon\KL(\gamma\mid\mu\otimes\nu)
\right\}.
\end{equation}
\end{definition}

\begin{theorem}[Entropic Kantorovich duality]
\label{thm:classical-entropic-ot-duality}
The problem \eqref{eq:classical-entropic-ot} has a unique minimizer and
\begin{align}\label{eq:classical-entropic-ot-dual}
\EOT_{\epsilon,c}(\mu,\nu)
=\epsilon+
\sup_{u,v\in C(K)}
\Bigg\{
&\int_Ku\d\mu+\int_Kv\d\nu\nonumber\\
&-\epsilon\int_{K\times K}
\exp\!\left(\frac{u(x)+v(y)-c(x,y)}{\epsilon}\right)
\d(\mu\otimes\nu)(x,y)
\Bigg\}.
\end{align}
\end{theorem}

The additive constant $\epsilon$ in
\eqref{eq:classical-entropic-ot-dual} is a consequence of the
relative-entropy convention \eqref{eq:kl-definition}; compare
\eqref{eq:entropy-shift}.  The equality between the primal infimum and
the dual supremum in \eqref{eq:classical-entropic-ot-dual} is an
established duality result for entropy-regularized optimal transport;
see \cite{marino2020optimal}.  This equality concerns the optimal
values and does not, by itself, assert that the dual supremum is
attained in $C(K)\times C(K)$.  Uniqueness of the primal optimizer
follows from the strict convexity of relative entropy on the convex
set $\Pi(\mu,\nu)$. 

Whenever the dual supremum is attained at $(u^*,v^*)$, the optimal plan satisfies
\begin{equation}\label{eq:classical-gibbs-relation}
\frac{\d\gamma_\epsilon}{\d(\mu\otimes\nu)}(x,y)
=
\exp\!\left(
\frac{u^*(x)+v^*(y)-c(x,y)}{\epsilon}
\right)
\qquad (\mu\otimes\nu)\text{-a.e.}
\end{equation}

\subsection{Fixed-support entropy-regularized barycenters}
\label{subsec:classical-entropic-barycenter}

There are several notions of entropy regularization for barycenters.  We use
the plan-regularized model with a fixed barycenter support measure, because it
keeps the regularizing references independent of the unknown barycenter and
is the model that will be combined with the hard capacities below.

Let $\eta\in\P(K)$ satisfy $\spt\eta=K$, and define
\begin{equation}\label{eq:reference-measure}
\xi_i:=\nu_i\otimes\eta,
\qquad i=1,\ldots,p.
\end{equation}

\begin{definition}[Fixed-support entropy-regularized barycenter]
\label{def:classical-entropic-barycenter}
For $\epsilon>0$, define
\begin{align}\label{eq:classical-entropic-barycenter}
\EBC_{\epsilon,\eta}
:=
\inf_{\substack{\nu\in\P(K)\\
\gamma_i\in\Pi(\nu_i,\nu)}}
\sum_{i=1}^p\lambda_i
\left[
\int_{K\times K}c_i\d\gamma_i
+\epsilon\KL(\gamma_i\mid\xi_i)
\right].
\end{align}
\end{definition}

\begin{theorem}[Fixed-support entropic barycenter duality]
\label{thm:classical-entropic-barycenter-duality}
The problem \eqref{eq:classical-entropic-barycenter} has a unique optimal
tuple of plans and hence a unique barycenter.  Moreover,
\begin{align}\label{eq:classical-entropic-barycenter-dual}
\EBC_{\epsilon,\eta}
=\epsilon+
\sup_{(\boldsymbol\phi,\boldsymbol\psi)\in\A_0}
\sum_{i=1}^p\lambda_i
\Bigg[
\int_K\phi_i\d\nu_i
-\epsilon\int_{K\times K}
\exp\!\left(
\frac{\phi_i(x)+\psi_i(y)-c_i(x,y)}{\epsilon}
\right)
\d\xi_i(x,y)
\Bigg].
\end{align}
Whenever the dual supremum is attained, the optimal plans satisfy
\begin{equation}\label{eq:classical-barycenter-gibbs-relation}
\frac{\d\gamma_{i,\epsilon}}{\d\xi_i}(x,y)
=
\exp\!\left(
\frac{\phi_i^*(x)+\psi_i^*(y)-c_i(x,y)}{\epsilon}
\right)
\qquad \xi_i\text{-a.e.}
\end{equation}
\end{theorem}

\begin{proof}
The strong duality and primal existence statements follow from the
fixed-support regularized barycenter theorem in \cite{li2020continuous}.  The
objective used there contains $\epsilon\rho$; the present objective differs by
the constant $\epsilon$ because of \eqref{eq:entropy-shift}.  The feasible set of plan tuples is convex, and the
sum of the fixed-reference relative entropies is strictly convex.  Hence the
optimal tuple, and therefore its common second marginal, is unique.  The
relation \eqref{eq:classical-barycenter-gibbs-relation} is the corresponding
Fenchel equality whenever dual optimizers exist.
\end{proof}

\begin{remark}[The reference measure is fixed]
\label{rem:fixed-reference}
The measure $\eta$ is chosen before the optimization.  Replacing $\eta$ in
\eqref{eq:reference-measure} by the unknown barycenter $\nu$ produces a
different regularized model and does not lead to the dual formula
\eqref{eq:classical-entropic-barycenter-dual}.  We use the fixed-reference
model throughout the remainder of the paper.
\end{remark}

\begin{remark}[Fixed- and variable-reference regularization]
\label{rem:fixed-variable-reference}
One may alternatively consider a regularization term of the form
\[
\KL(\gamma_i\mid\nu_i\otimes\nu),
\]
where $\nu$ is the unknown barycenter.  This produces a
variable-reference model that is different from
\eqref{eq:classical-entropic-barycenter}.  In the present paper,
$\eta$ is prescribed before the optimization and
\[
\xi_i=\nu_i\otimes\eta
\]
remains fixed.  This choice is essential for the convex dual
formulation used below and for representing each hard capacity as
\[
\widetilde\gamma_i=h_i\xi_i.
\]
No equivalence between the fixed- and variable-reference models is
claimed.
\end{remark}

\subsection{Entropy-regularized capacity-constrained transport}
\label{subsec:entropic-capacity-transport}

Let $\xi\in\P(K\times K)$ and let
$h\in L^\infty_+(K\times K,\xi)$, where $L^\infty_+$ denotes the cone
of $\xi$-almost everywhere nonnegative essentially bounded functions.  The
finite measure $\widetilde\gamma:=h\xi$ is interpreted as a hard transport
capacity.

\begin{definition}[Entropy-regularized capacity-constrained transport]
\label{def:eccot}
For $\mu,\nu\in\P(K)$, $c\in C(K\times K)$, and $\epsilon>0$, define
\begin{equation}\label{eq:eccot-primal}
\ECCOT_{\epsilon,c}^{h,\xi}(\mu,\nu)
:=
\inf_{\substack{\gamma\in\Pi(\mu,\nu)\\
0\le\gamma\le h\xi}}
\left\{
\int_{K\times K}c\d\gamma
+\epsilon\KL(\gamma\mid\xi)
\right\},
\end{equation}
with value $+\infty$ when the admissible set is empty.
\end{definition}

The interaction between entropy and the upper bound is encoded by the
Fenchel conjugate of the entropy restricted to the interval $[0,h]$.

\begin{definition}[Capped-exponential conjugate]
\label{def:capped-conjugate}
For $\epsilon>0$, $h\ge0$, and $s\in\R$, define
\begin{equation}\label{eq:capped-conjugate-definition}
\Phi_{\epsilon,h}(s)
:=
\sup_{0\le q\le h}
\bigl\{sq-\epsilon\rho(q)\bigr\},
\qquad
\rho(q)=q\log q-q.
\end{equation}
Equivalently,
$\Phi_{\epsilon,h}=(\epsilon\rho+I_{[0,h]})^*$, where
$I_{[0,h]}$ is $0$ on $[0,h]$ and $+\infty$ outside this interval.
\end{definition}

\begin{proposition}[Explicit formula and derivative]
\label{prop:capped-conjugate}
If $h=0$, then $\Phi_{\epsilon,0}(s)=0$.  If $h>0$, then
\begin{equation}\label{eq:capped-conjugate-explicit}
\Phi_{\epsilon,h}(s)
=
\begin{cases}
\epsilon\exp(s/\epsilon),
& s\le\epsilon\log h,\\[0.4em]
hs-\epsilon h\log h+\epsilon h,
& s>\epsilon\log h.
\end{cases}
\end{equation}
Moreover, $\Phi_{\epsilon,h}$ is convex and continuously differentiable,
with
\begin{equation}\label{eq:capped-conjugate-derivative}
\frac{\d}{\d s}\Phi_{\epsilon,h}(s)
=\min\bigl\{h,\exp(s/\epsilon)\bigr\}.
\end{equation}
\end{proposition}

\begin{proof}
For $q>0$, the derivative of
$q\mapsto sq-\epsilon(q\log q-q)$ is $s-\epsilon\log q$.
The unconstrained maximizer is $q=\exp(s/\epsilon)$.  If this value does not
exceed $h$, substitution gives the first branch in
\eqref{eq:capped-conjugate-explicit}; otherwise the maximum is attained at
$q=h$.  The values and first derivatives agree at $s=\epsilon\log h$.
\end{proof}

\begin{theorem}[Entropic duality with a hard capacity]
\label{thm:eccot-duality}
Assume that the admissible set in \eqref{eq:eccot-primal} is nonempty.  Then
the infimum is attained at a unique plan and
\begin{align}\label{eq:eccot-dual}
\ECCOT_{\epsilon,c}^{h,\xi}(\mu,\nu)
=\epsilon+
\sup_{u,v\in C(K)}
\Bigg\{
&\int_Ku\d\mu+\int_Kv\d\nu\nonumber\\
&-\int_{K\times K}
\Phi_{\epsilon,h(x,y)}
\bigl(u(x)+v(y)-c(x,y)\bigr)
\d\xi(x,y)
\Bigg\}.
\end{align}
\end{theorem}

\begin{proof}
Write $\gamma=q\xi$.  The feasible density set is weak-* compact, the cost is
weak-* continuous, and the entropy is weak-* lower semicontinuous and strictly
convex.  Thus the primal problem has a unique minimizer.  To derive the dual,
the marginal constraints have the indicator
representation
\[
\sup_{u,v\in C(K)}
\left\{
\int_Ku\d\mu+\int_Kv\d\nu
-\int_{K\times K}(u(x)+v(y))q(x,y)\d\xi(x,y)
\right\}.
\]
The order interval $\{q\in L^\infty(\xi):0\le q\le h\}$ is weak-* compact.
The map $q\mapsto\int\rho(q)\d\xi$ is weak-* lower semicontinuous on this
interval, since
\[
\int\rho(q)\d\xi
=
\sup_{a\in L^\infty(\xi)}
\left\{
\int aq\d\xi-\int e^a\d\xi
\right\}.
\]
Sion's minimax theorem therefore permits interchange of the minimum and the
supremum.  The pointwise minimization in $q$ is exactly the negative of
\eqref{eq:capped-conjugate-definition}.  The minimizing density
$\min\{h,\exp((u+v-c)/\epsilon)\}$ is measurable, so the pointwise minimum may
be integrated.  Restoring the constant in
\eqref{eq:entropy-shift} gives \eqref{eq:eccot-dual}.
\end{proof}

Entropy-based numerical methods for constrained two-marginal transport are
studied in \cite{wu2023double,tang2024sinkhorn}.  The preceding theorem is
included to record, in the fixed-reference notation needed below, the exact
hard-capacity conjugate that passes to the barycenter problem.

The capped conjugate also has the multiplier representation
\begin{equation}\label{eq:multiplier-representation}
-\Phi_{\epsilon,h}(s)
=
\sup_{w\le0}
\left\{
wh-\epsilon\exp\!\left(\frac{s+w}{\epsilon}\right)
\right\}.
\end{equation}
For $h>0$, the maximizing scalar is
\begin{equation}\label{eq:optimal-capacity-multiplier}
w^*(s,h)=\min\{0,\epsilon\log h-s\}.
\end{equation}
These identities follow by differentiating the right-hand side with respect
to $w$ and checking the boundary $w=0$.

\subsection{Entropy-regularized capacity-constrained barycenters}
\label{subsec:entropic-primal}

We now return to the fixed references $\xi_i=\nu_i\otimes\eta$ from
\eqref{eq:reference-measure} and impose the following compatibility condition.

\begin{assumption}[Capacity densities]
\label{ass:capacity-densities}
For each $i$, there exists
$h_i\in L^\infty_+(K\times K,\xi_i)$ such that
\begin{equation}\label{eq:capacity-density}
\widetilde\gamma_i=h_i\xi_i.
\end{equation}
\end{assumption}

This is an additional restriction on the capacities.  It is automatic in a
finite discrete model when $\xi_i$ assigns positive mass to every cell on
which $\widetilde\gamma_i$ is nonzero.  It is not automatic in a continuous
model and excludes capacities singular with respect to
$\nu_i\otimes\eta$.  The assumption is imposed so that the entropy and the
hard capacity can be expressed in the same density variable.

\begin{definition}[Entropy-regularized capacity-constrained barycenter]
\label{def:eccbc}
For $\epsilon>0$, define
\begin{align}\label{eq:eccbc-primal}
\ECCBC_{\epsilon,\eta}
&:=
\inf_{\nu\in\P(K)}
\sum_{i=1}^p\lambda_i
\ECCOT_{\epsilon,c_i}^{h_i,\xi_i}(\nu_i,\nu)\nonumber\\
&=
\inf_{\substack{\nu\in\P(K)\\
\gamma_i\in\Pi^{\widetilde\gamma_i}(\nu_i,\nu)}}
\sum_{i=1}^p\lambda_i
\left[
\int_{K\times K}c_i\d\gamma_i
+\epsilon\KL(\gamma_i\mid\xi_i)
\right].
\end{align}
A minimizer $\nu_\epsilon$ is called an entropy-regularized
capacity-constrained barycenter.
\end{definition}

Every admissible plan has the form $\gamma_i=q_i\xi_i$ with
$0\le q_i\le h_i$.  Its marginal constraints are
\begin{align}
\int_Kq_i(x,y)\d\eta(y)&=1
&&\text{for $\nu_i$-a.e. }x,
\label{eq:density-source-marginal}\\
\int_Kq_i(x,y)\d\nu_i(x)&=r(y)
&&\text{for $\eta$-a.e. }y,
\label{eq:density-common-marginal}
\end{align}
where $r=\d\nu/\d\eta$ is independent of $i$.

\begin{theorem}[Existence and uniqueness]
\label{thm:eccbc-existence-uniqueness}
Suppose \Cref{ass:capacity-densities} holds and the admissible set in
\eqref{eq:eccbc-primal} is nonempty.  Then the regularized problem has a
unique optimal tuple of plans and, consequently, a unique barycenter
$\nu_\epsilon$.
\end{theorem}

\begin{proof}
The feasible set of plan tuples is weakly compact: the domination
$0\le\gamma_i\le\widetilde\gamma_i$ gives tightness, and the domination and
marginal conditions are closed under weak convergence.  The transport costs
are weakly continuous because the $c_i$ are continuous on the compact set
$K\times K$, while relative entropy with respect to a fixed reference is
weakly lower semicontinuous \cite[Lemma~1.4.3]{dupuis2011weak}.  Hence a
minimizer exists.  The feasible set is
convex and the weighted sum of the fixed-reference relative entropies is
strictly convex, so the optimal tuple is unique.  Its common second marginal
is therefore unique.
\end{proof}

\subsection{Strong duality and capacity-clipped optimality relations}
\label{subsec:eccbc-duality}

The same potential class $\A_0$ used in the unregularized barycenter dual
encodes the source marginals and the common second marginal.

\begin{theorem}[Strong duality for the entropy-regularized
capacity-constrained barycenter]
\label{thm:eccbc-duality}
Suppose \Cref{ass:capacity-densities} holds and the admissible set in
\eqref{eq:eccbc-primal} is nonempty.  Then
\begin{align}\label{eq:eccbc-dual-capped}
\ECCBC_{\epsilon,\eta}
=\epsilon+
\sup_{(\boldsymbol\phi,\boldsymbol\psi)\in\A_0}
\sum_{i=1}^p\lambda_i
\Bigg[
\int_K\phi_i\d\nu_i
-\int_{K\times K}
\Phi_{\epsilon,h_i(x,y)}
\bigl(\phi_i(x)+\psi_i(y)-c_i(x,y)\bigr)
\d\xi_i(x,y)
\Bigg].
\end{align}
\end{theorem}

\begin{proof}
For each $i$, set
\[
\mathcal D_i^h
:=
\{q\in L^\infty(K\times K,\xi_i):0\le q\le h_i\}.
\]
This is a convex weak-* compact subset of $L^\infty(\xi_i)$.  For
$\boldsymbol q\in\prod_i\mathcal D_i^h$, the marginal and common-target
constraints have the indicator representation
\begin{align}\label{eq:entropic-barycentric-constraint-indicator}
\mathcal I_\xi(\boldsymbol q)
:=\sup_{(\boldsymbol\phi,\boldsymbol\psi)\in\A_0}
\sum_{i=1}^p\lambda_i
\left[
\int_K\phi_i\d\nu_i
-\int_{K\times K}
(\phi_i(x)+\psi_i(y))q_i(x,y)\d\xi_i(x,y)
\right].
\end{align}
The same separation argument as in
\eqref{eq:barycentric-constraint-indicator} shows that
$\mathcal I_\xi$ is zero precisely for tuples with the prescribed source
marginals and a common second marginal, and is $+\infty$ otherwise.  Applying
\eqref{eq:entropy-shift} on the feasible tuples therefore gives
\begin{align*}
\ECCBC_{\epsilon,\eta}-\epsilon
=
\min_{\boldsymbol q\in\prod_i\mathcal D_i^h}
\sup_{(\boldsymbol\phi,\boldsymbol\psi)\in\A_0}
\sum_{i=1}^p\lambda_i
\Bigg[
\int_K\phi_i\d\nu_i
+\int_{K\times K}
\Bigl((c_i-\phi_i-\psi_i)q_i+\epsilon\rho(q_i)\Bigr)
\d\xi_i
\Bigg].
\end{align*}
For fixed potentials, the functional is convex and weak-* lower
semicontinuous in $\boldsymbol q$.  The lower semicontinuity follows, for
example, from the variational representation
\[
\int\rho(q_i)\d\xi_i
=
\sup_{a\in L^\infty(\xi_i)}
\left\{
\int aq_i\d\xi_i-\int e^a\d\xi_i
\right\}.
\]
For fixed $\boldsymbol q$, the functional is affine and continuous in the
potentials.  Sion's minimax theorem therefore gives
\[
\min_{\boldsymbol q}\sup_{(\boldsymbol\phi,\boldsymbol\psi)}
=
\sup_{(\boldsymbol\phi,\boldsymbol\psi)}\min_{\boldsymbol q}.
\]
The inner minimization is pointwise, and
\[
\min_{0\le q\le h_i}
\left\{(c_i-\phi_i-\psi_i)q+\epsilon\rho(q)\right\}
=
-\Phi_{\epsilon,h_i}(\phi_i+\psi_i-c_i).
\]
The measurable minimizer is
$\min\{h_i,\exp((\phi_i+\psi_i-c_i)/\epsilon)\}$, which justifies integrating
the pointwise minimum.  Restoring the additive constant proves \eqref{eq:eccbc-dual-capped}.
\end{proof}

\begin{remark}[Dual attainment]
\label{rem:continuous-dual-attainment}
\Cref{thm:eccbc-duality} proves equality of the primal and dual values.  It
does not by itself prove that the supremum is attained by continuous
potentials, and no such continuous-attainment claim is made here.  In a
finite-dimensional discrete problem, standard convex duality gives dual
attainment under a relative-interior feasibility condition for the box
$0\le q_i\le h_i$.
\end{remark}

\begin{corollary}[Capacity-multiplier form]
\label{cor:eccbc-multiplier-dual}
Assume, in addition, that $h_i\in C(K\times K)$ and
$h_i\ge\underline h_i>0$ for every $i$.  Then
\begin{align}\label{eq:eccbc-dual-multiplier}
\ECCBC_{\epsilon,\eta}
=\epsilon+
\sup\sum_{i=1}^p\lambda_i
\Bigg[
&\int_K\phi_i\d\nu_i
+\int_{K\times K}w_i\d\widetilde\gamma_i\nonumber\\
&-\epsilon\int_{K\times K}
\exp\!\left(
\frac{\phi_i(x)+\psi_i(y)+w_i(x,y)-c_i(x,y)}{\epsilon}
\right)
\d\xi_i(x,y)
\Bigg],
\end{align}
where the supremum is over $\phi_i,\psi_i\in C(K)$ and
$w_i\in C(K\times K)$ satisfying
$w_i\le0$ and $\sum_i\lambda_i\psi_i=0$.
\end{corollary}

\begin{proof}
Apply \eqref{eq:multiplier-representation} pointwise with
$s=\phi_i(x)+\psi_i(y)-c_i(x,y)$ and use
$\widetilde\gamma_i=h_i\xi_i$.  Under the additional assumptions, the
maximizer
$w_i=\min\{0,\epsilon\log h_i-s\}$ is continuous, so the pointwise
supremum is realized in $C(K\times K)$.
\end{proof}

\begin{theorem}[Capacity-clipped primal--dual relations]
\label{thm:eccbc-kkt}
Assume the hypotheses of \Cref{thm:eccbc-duality}, and suppose that
$(\phi_i^*,\psi_i^*)_{i=1}^p$ is a dual maximizer in
\eqref{eq:eccbc-dual-capped}.  Let
$(\gamma_{i,\epsilon})_{i=1}^p$ be the unique primal optimizer and write
$\gamma_{i,\epsilon}=q_{i,\epsilon}\xi_i$ and
$\nu_\epsilon=r_\epsilon\eta$.  Then
\begin{equation}\label{eq:clipped-gibbs-density}
q_{i,\epsilon}(x,y)
=
\min\left\{
 h_i(x,y),
 \exp\!\left(
 \frac{\phi_i^*(x)+\psi_i^*(y)-c_i(x,y)}{\epsilon}
 \right)
\right\}
\qquad\xi_i\text{-a.e.}
\end{equation}
and
\begin{align}
\int_Kq_{i,\epsilon}(x,y)\d\eta(y)&=1
&&\nu_i\text{-a.e.},
\label{eq:clipped-schrodinger-row}\\
\int_Kq_{i,\epsilon}(x,y)\d\nu_i(x)&=r_\epsilon(y)
&&\eta\text{-a.e.},
\label{eq:clipped-schrodinger-column}\\
\sum_{i=1}^p\lambda_i\psi_i^*(y)&=0
&&\text{for every }y\in K.
\label{eq:clipped-schrodinger-barycenter}
\end{align}
On the set $E_i:=\{h_i>0\}$, define the measurable capacity field
\begin{equation}\label{eq:optimal-capacity-field}
w_i^*(x,y)
:=
\min\left\{
0,
\epsilon\log h_i(x,y)
-\phi_i^*(x)-\psi_i^*(y)+c_i(x,y)
\right\}.
\end{equation}
Then
\begin{equation}\label{eq:capacity-complementarity}
w_i^*\le0,
\qquad
q_{i,\epsilon}\le h_i,
\qquad
w_i^*(h_i-q_{i,\epsilon})=0
\quad\xi_i\text{-a.e. on }E_i.
\end{equation}
On $\{h_i=0\}$, one has $q_{i,\epsilon}=0$.
\end{theorem}

\begin{proof}
Equality of the primal and dual values forces $q_{i,\epsilon}$ to minimize
the pointwise expression used in the proof of
\Cref{thm:eccbc-duality}.  Strict convexity in $q$ and
\eqref{eq:capped-conjugate-derivative} give
\eqref{eq:clipped-gibbs-density}.  The remaining equations are the primal
marginal constraints and the barycentric dual constraint.  On $E_i$, formula \eqref{eq:optimal-capacity-field} is the optimizer in
\eqref{eq:optimal-capacity-multiplier}, and
\eqref{eq:capacity-complementarity} follows immediately.
\end{proof}

\begin{remark}[Capacity saturation and algorithms]
\label{rem:capacity-saturation}
On $E_i=\{h_i>0\}$, where
$\phi_i^*(x)+\psi_i^*(y)-c_i(x,y)<\epsilon\log h_i(x,y)$,
the capacity is inactive and the density has the usual Gibbs form.  At or
above the threshold, the density equals $h_i$ and the capacity is saturated.
On $\{h_i=0\}$, the density vanishes.  This clipped relation suggests a
capacity-aware scaling procedure, but no iterative algorithm or convergence
theorem is proved in this paper.
\end{remark}

\section{The zero-regularization limit and entropy selection}
\label{sec:zero-limit}

Throughout this section, retain \Cref{ass:capacity-densities} and let
$\CCBC$ denote the unregularized value with the same capacities
$\widetilde\gamma_i=h_i\xi_i$.

\begin{proposition}[Convergence of the dual penalty]
\label{prop:dual-penalty-limit}
For every $h\ge0$ and $s\in\R$,
\begin{equation}\label{eq:dual-penalty-limit}
\lim_{\epsilon\downarrow0}\Phi_{\epsilon,h}(s)
=h\pos{s}.
\end{equation}
Thus the capped-exponential penalty is a continuously differentiable
approximation of the positive-part capacity penalty.
\end{proposition}

\begin{proof}
If $h=0$, the claim is immediate.  Assume $h>0$.  If $s<0$, then for all
sufficiently small $\epsilon$ the first branch of
\eqref{eq:capped-conjugate-explicit} applies and
$\epsilon e^{s/\epsilon}\to0$.  If $s=0$, either branch converges to zero.
If $s>0$, then for all sufficiently small $\epsilon$ the second branch
applies and
\[
\Phi_{\epsilon,h}(s)
=hs-\epsilon h\log h+\epsilon h
\longrightarrow hs.
\]
\end{proof}

Since $\widetilde\gamma_i=h_i\xi_i$, the unregularized dual can be written as
\begin{align}\label{eq:unregularized-positive-part-dual-density}
\CCBC
=
\sup_{(\boldsymbol\phi,\boldsymbol\psi)\in\A_0}
\sum_{i=1}^p\lambda_i
\left[
\int_K\phi_i\d\nu_i
-\int_{K\times K}
h_i(x,y)\pos{\phi_i(x)+\psi_i(y)-c_i(x,y)}
\d\xi_i(x,y)
\right].
\end{align}
Thus \Cref{prop:dual-penalty-limit} gives the pointwise connection between
the regularized and unregularized dual penalties.

Let $\mathcal S_0$ denote the set of unregularized optimal plan tuples:
\begin{equation}\label{eq:unregularized-optimal-set}
\mathcal S_0
:=
\left\{
(\gamma_1,\ldots,\gamma_p):
\begin{array}{l}
\text{there exists }\nu\in\P(K)\text{ such that }
\gamma_i\in\Pi^{\widetilde\gamma_i}(\nu_i,\nu),\\[0.2em]
\displaystyle
\sum_{i=1}^p\lambda_i\int_{K\times K}c_i\d\gamma_i
=\CCBC
\end{array}
\right\}.
\end{equation}
For a feasible tuple $\boldsymbol\gamma=(\gamma_1,\ldots,\gamma_p)$, set
\begin{equation}\label{eq:weighted-entropy-functional}
\mathcal H(\boldsymbol\gamma)
:=
\sum_{i=1}^p\lambda_i\KL(\gamma_i\mid\xi_i).
\end{equation}

\begin{theorem}[Zero regularization and entropy selection]
\label{thm:eccbc-zero-limit}
Assume \Cref{ass:capacity-densities} and that the common feasible set is
nonempty.  There exists a unique tuple
$\overline{\boldsymbol\gamma}
=(\overline\gamma_1,\ldots,\overline\gamma_p)\in\mathcal S_0$
minimizing $\mathcal H$ over $\mathcal S_0$.  If $\overline\nu$ is its
common second marginal, then
\begin{equation}\label{eq:value-rate}
0
\le
\ECCBC_{\epsilon,\eta}-\CCBC
\le
\epsilon\mathcal H(\overline{\boldsymbol\gamma})
\qquad\text{for every }\epsilon>0,
\end{equation}
and, as $\epsilon\downarrow0$,
\begin{equation}\label{eq:entropy-selected-convergence}
\gamma_{i,\epsilon}\rightharpoonup\overline\gamma_i
\quad\text{for every }i,
\qquad
\nu_\epsilon\rightharpoonup\overline\nu.
\end{equation}
In particular, the whole regularized family converges to the unique
entropy-minimizing unregularized optimal tuple.
\end{theorem}

\begin{proof}
The feasible set of plan tuples is weakly compact and convex.  Since each
$c_i$ is continuous, $\mathcal S_0$ is a nonempty compact convex subset of
that feasible set.  Moreover, $0\le\gamma_i\le h_i\xi_i$ with
$h_i\in L^\infty(\xi_i)$, so $\mathcal H$ is finite on the feasible set.
It is weakly lower semicontinuous and strictly convex on plan tuples.
Therefore it has a unique minimizer $\overline{\boldsymbol\gamma}$ on $\mathcal S_0$.

Let
$\boldsymbol\gamma_\epsilon
=(\gamma_{1,\epsilon},\ldots,\gamma_{p,\epsilon})$
be the regularized optimizer and write
\[
\mathcal C(\boldsymbol\gamma)
:=
\sum_{i=1}^p\lambda_i\int_{K\times K}c_i\d\gamma_i.
\]
Optimality gives
\[
\mathcal C(\boldsymbol\gamma_\epsilon)
+\epsilon\mathcal H(\boldsymbol\gamma_\epsilon)
\le
\mathcal C(\overline{\boldsymbol\gamma})
+\epsilon\mathcal H(\overline{\boldsymbol\gamma})
=
\CCBC+\epsilon\mathcal H(\overline{\boldsymbol\gamma}).
\]
Because relative entropy between probability measures is nonnegative,
$\ECCBC_{\epsilon,\eta}\ge\CCBC$.  In addition,
$\mathcal C(\boldsymbol\gamma_\epsilon)\ge\CCBC$ in the preceding
optimality inequality, which gives the upper bound in \eqref{eq:value-rate}
and also
\begin{equation}\label{eq:entropy-bound-regularized-plans}
\mathcal H(\boldsymbol\gamma_\epsilon)
\le
\mathcal H(\overline{\boldsymbol\gamma}).
\end{equation}

Let $\epsilon_n\downarrow0$.  Compactness gives a subsequence, not relabeled,
with $\boldsymbol\gamma_{\epsilon_n}\rightharpoonup\boldsymbol\gamma^*$.
The limiting tuple is feasible.  By \eqref{eq:value-rate} and continuity of
the transport costs,
$\mathcal C(\boldsymbol\gamma^*)=\CCBC$, so
$\boldsymbol\gamma^*\in\mathcal S_0$.  Lower semicontinuity of
$\mathcal H$ and \eqref{eq:entropy-bound-regularized-plans} yield
\[
\mathcal H(\boldsymbol\gamma^*)
\le
\liminf_{n\to\infty}
\mathcal H(\boldsymbol\gamma_{\epsilon_n})
\le
\mathcal H(\overline{\boldsymbol\gamma}).
\]
By uniqueness of the entropy minimizer,
$\boldsymbol\gamma^*=\overline{\boldsymbol\gamma}$.  Every weak cluster
point is therefore the same, which proves convergence of the entire family.
The convergence of the common second marginals follows by continuity of the
push-forward under the coordinate projection.
\end{proof}

\begin{remark}[Dependence on the reference measure]
\label{rem:selection-reference-dependence}
The entropy-selected tuple
$\overline{\boldsymbol\gamma}$ generally depends on the prescribed
support measure $\eta$, since the selecting functional is
\[
\mathcal H(\boldsymbol\gamma)
=
\sum_{i=1}^p
\lambda_i
\KL(\gamma_i\mid\nu_i\otimes\eta).
\]
Thus \Cref{thm:eccbc-zero-limit} gives a unique selected limit for
each fixed choice of $\eta$; it does not assert that the selected
unregularized optimizer is independent of the regularizing
reference.
\end{remark}

\section{Discussion and future directions}
\label{sec:discussion}

The capacity-constrained barycenter extends the classical Wasserstein
barycenter by placing an upper bound on each transport plan connecting
an input measure to the common barycenter.  This is useful when the
barycenter must be formed under local restrictions and not every
movement of mass is freely available.  A capacity may limit how much
mass can be transported between two regions, or it may completely
exclude certain connections.

Such restrictions arise naturally in problems involving transport
networks, logistics, resource allocation, and matching with quotas.
For example, a transport network may have limited flow on some routes,
a distribution system may have bounded supply between particular
locations, and a matching problem may restrict the number of
assignments between two groups.  In these settings, the ordinary
Wasserstein barycenter may use transport plans that are mathematically
optimal but practically impossible.  The capacity-constrained
barycenter provides an average that respects the prescribed
limitations.

The results of this paper show that adding these constraints does not
destroy the main mathematical structure of the barycenter problem.
Under the stated feasibility assumptions, a capacity-constrained
barycenter exists, and its optimal value has a dual representation.
The dual variables associated with the capacities help describe where
the restrictions influence the optimal transport plans.

Entropy regularization provides a more tractable approximation of the
constrained problem.  For every positive regularization parameter, the
regularized problem has a unique optimal tuple of transport plans.
Whenever the corresponding dual problem has a maximizer, the optimal
density has a clipped Gibbs form: it follows the usual exponential
formula where the capacity is inactive and equals the prescribed
capacity where the constraint is saturated.  This gives a direct way
to distinguish unrestricted transport regions from saturated regions.

The zero-regularization result also explains the relation between the
regularized and unregularized models.  As the regularization parameter
tends to zero, the regularized optimal plans converge to an optimal
tuple for the original capacity-constrained problem.  When the
unregularized problem has several optimal tuples, the regularization
selects the one with minimum relative entropy with respect to the
chosen reference measures.  This selected solution may depend on the
reference measure used in the entropy term.

Several questions remain open.  One important direction is to prove
existence and regularity of dual maximizers for the continuous
capacity-constrained entropy problem.  Another is to develop and
analyze numerical methods, such as capacity-aware scaling or
Sinkhorn-type procedures, based on the clipped Gibbs relation.
Numerical examples would also help illustrate how the barycenter
changes as the capacities are tightened or relaxed.

Further work may study stability with respect to changes in the input
measures and capacity measures, as well as quantitative convergence
rates as the regularization parameter tends to zero.  It would also be
useful to extend the theory to noncompact spaces, more general costs,
unbalanced transport, and multimarginal formulations.  Finally,
application-specific choices of the capacities and of the entropy
reference measure deserve a separate investigation.
\vspace{1cm}

\appendix
\textbf{Appendix}
\section{A second-moment estimate on \texorpdfstring{$\R^d$}{R\textasciicircum d}}\label{sec:appendix-moments}

The compact-domain proof of \Cref{thm:ccbc-existence} is particularly direct.  On
$\R^d$, one may instead use the following estimate from the original direct-method argument.

\begin{proposition}[Uniform second-moment bound]\label{prop:moment-bound}
Let $\mu\in\P_2(\R^d)$, let
$\{\nu_n\}_{n\in\N}\subset\P_2(\R^d)$, and let
$\widetilde\gamma\in\M_+(\R^d\times\R^d)$ be such that
$\Pi^{\widetilde\gamma}(\mu,\nu_n)\neq\varnothing$ for every $n$.  If
\[
\CCOT_{\frac12|x-y|^2}^{\widetilde\gamma}(\mu,\nu_n)
\le M
\qquad\text{for every }n,
\]
then
\[
\sup_n\int_{\R^d}|y|^2\d\nu_n(y)<\infty.
\]
\end{proposition}

\begin{proof}
The pointwise inequality
\[
-\frac12|x|^2+\frac14|y|^2
\le\frac12|x-y|^2
\]
follows from
$\frac12|x-y|^2+\frac12|x|^2-\frac14|y|^2
=|x-y/2|^2$.  Integrating it against any
$\gamma_n\in\Pi^{\widetilde\gamma}(\mu,\nu_n)$ gives
\[
\frac14\int_{\R^d}|y|^2\d\nu_n(y)
\le
\int_{\R^d\times\R^d}\frac12|x-y|^2\d\gamma_n(x,y)
+\frac12\int_{\R^d}|x|^2\d\mu(x).
\]
Taking the infimum over admissible $\gamma_n$ yields
\[
\frac14\int_{\R^d}|y|^2\d\nu_n(y)
\le
\CCOT_{\frac12|x-y|^2}^{\widetilde\gamma}(\mu,\nu_n)
+\frac12\int_{\R^d}|x|^2\d\mu(x),
\]
and the right-hand side is uniformly bounded.
\end{proof}

Combined with Chebyshev's inequality, Prokhorov's theorem, closure of the capacity inequalities, and weak lower semicontinuity, \Cref{prop:moment-bound} yields the corresponding existence argument on noncompact domains whenever the capacities and the feasible class permit the limiting passage.

\end{document}